\documentclass[a4paper,12pt]{article}
\usepackage[text={6.5in,9.5in},centering]{geometry}
\usepackage{graphicx,url,subfig}
\RequirePackage[OT1]{fontenc}
\RequirePackage{amsthm,amsmath,amssymb,amscd}
\RequirePackage[colorlinks,citecolor=blue,urlcolor=blue]{hyperref}
\usepackage{enumerate}
\usepackage{datetime}
\usepackage[normalem]{ulem} 
\usepackage{soul} 
\makeatother
\numberwithin{equation}{section}
\allowdisplaybreaks
\usepackage{dsfont}
\usepackage{authblk}
\usepackage{mathtools}

\theoremstyle{plain}
\newtheorem{theorem}{Theorem}[section]

\theoremstyle{definition}

\newtheorem{remark}[theorem]{Remark}

\newcommand{\E}{\mathbb{E}}

\newcommand{\Norm}[1]{\left\|  #1 \right\|}

\newcommand{\R}{\mathbb{R}}

\newcommand{\NormO}[2]{\left\|  #1 \right\|_{L^{#2}\left(\Omega\right)}}

\newcommand{\lr}[1]{\left(  #1 \right)}

\newcommand{\inner}[2]{\left\langle #1 , #2 \right\rangle}

\title{A Khintchine type inequality for non-independent random variables}
\author[1]{Jingyu Huang\thanks{j.huang.4@bham.ac.uk}}
\author[1]{Wenxuan Tao\thanks{wxt399@student.bham.ac.uk}}
\affil[1]{School of Mathematics, University of Birmingham,
Birmingham,
B15 2TT,
United Kingdom}

\begin{document}
\date{}
\maketitle
\begin{abstract}
In this short note we give a Khintchine-type inequality involving non-independent random variables. The basic setting is inspired by Breuer-Major theorem \cite{BM} and the proof is based on a hypercontractivity argument. 

\vspace{5mm}

\noindent\textbf{MSC2020 subject classification:} 60G15.
\end{abstract}
\section{Introduction}
This short note concerns the Khintchine-type inequality but for non-independent random variables. The Khintchine inequality, originally studied in \cite{Khintchine}, says that for a sequence of i.i.d. Rademacher random variables $\beta_i$, $i=1,2,\dots, N$ and any $x_i$, $i = 1,2,\dots, N$, $p,q>0$, there is a constant $C_{p,q}>0$ (which does not depend on $x_i$'s and $N$) such that 
\begin{equation}
   \left\|\sum_{i=1}^N x_i \beta_i \right\|_{L^p(\Omega)}\leq C_{p,q}  \left\|\sum_{i=1}^N x_i \beta_i \right\|_{L^q(\Omega)}\,.
\end{equation}
There has been a lot of research in generalizing this inequality and also finding the best constant $C_{p,q}$, see e.g.\cite{Havrilla}. The Khintchine inequality also finds important applications in other areas, e.g, the Littlewood-Paley theory in harmonic analysis, see Section 8.2 in \cite{MS}. 

As far as the authors know, all of these Khintchine type inequalities assume that the random variables are independent. We try to explore this theorem for non-independent random variables. The basic idea is inspired by the celebrated Breuer-Major theorem \cite{BM}, which concerns the central limit theorem for sums of non-independent random variables. To state the Breuer-Major theorem, let us first introduce the following setting, see e.g., \cite[Section 7.2]{nourdin.peccati:12:normal}. 

\begin{enumerate}
    \item $X_1,X_2, \dots,$ is a sequence of centered Gaussian random variables with covariance\\ $\E\lr{X_kX_l}=\rho(k-l)$, here $\rho(0)=1$ and $\rho$ is not identically zero.
    \item Let $f$ be a deterministic function with Hermite rank $d$, i.e.,  $f$ can be written as
    \begin{equation}\label{E:f=Hermite}
        f(x)=\sum_{q=d}^\infty a_qH_q(x)\,
    \end{equation}
    with $a_d\neq 0$, and $H_q(x)$ is the $q$-th Hermite polynomial defined by 
    \begin{equation}\label{E:Hermite_Poly}
        H_q(x) := (-1)^q e^{x^2/2}\frac{d^q}{dx^q} e^{-x^2/2}\,.
    \end{equation}
    \item For each $k \geq 1$, define 
    \begin{equation}\label{E:Y=f(X)}
        Y_k=f(X_k)\,.
    \end{equation}
    Then, $\mathrm{Var}\lr{Y_k}< \infty$ for some $\sigma$ and that $\mathrm{Cov}\lr{Y_k,Y_l}\ne 0$ for some $k\ne l$.
\end{enumerate}

By~\cite[Proposition 7.2.3]{nourdin.peccati:12:normal}, there exist a real separable Hilbert space $\mathfrak{H}$ and an isonormal Gaussian process $X$ over $\mathfrak{H}$, written $\{X(h): h \in \mathfrak{H}\}$ with the property that there exists a set $E = \{\varepsilon_k, k \in \mathbb{Z}\}$ such that (i) $E$ generates $\mathfrak{H}$, (ii) $\langle \varepsilon_l, \varepsilon_k\rangle_{\mathfrak{H}} = \rho(l-k)$ for each $k, l \in \mathbb{Z}$; (iii) $X_k=X(\varepsilon_k)$.  Thus, $\text{Cov}(X_k,X_l)= \inner{\varepsilon_k}{\varepsilon_l}_{\mathfrak{H}}=\rho(k-l)$, $\text{Var}(X_k) = \rho(0)=1$ and $\rho(k)\le 1$ for all $ k , l\geq 1$.
The Breuer-Major theorem states that under the assumptions 
\begin{equation}\label{E:rho_summable}
    \sum_{k\in \mathbb{Z}}|\rho(k)|^d<\infty
   \end{equation}
{and}
\begin{equation}\label{E:sigma_summable}
   \sigma^2 := \sum_{q=d}^{\infty} q! a_q^2 \sum_{k \in \mathbb{Z}}\rho(k)^q < \infty\,,
\end{equation}
the random variables 
\begin{equation}\label{E:B-M_variables}
    \frac{1}{\sqrt{n}}\sum_{k=1}^n f(X_k)  
\end{equation}
    converge in distribution to a centered normal random variable with variance $\sigma^2$ as $n\to \infty$. 


Thus the Breuer-Major theorem says that if either 
$\rho$ has fast decay, or 
$f$ has large Hermite rank, 
then \eqref{E:B-M_variables} will satisfy the central limit theorem. The large Hermite rank of $f$ in $Y_k = f(X_k)$ can make the new variables $Y_k$ have smaller covariances than the original $X_k$'s. 

Motivated by the mechanism for attenuating dependence underlying the Breuer--Major theorem, we consider the same framework, but restrict ourselves to a finite family $\{X_k\}_{k=1}^N$, where $X_k = X(\varepsilon_k)$ for all $k$. The normalized partial sums appearing in~\eqref{E:B-M_variables} may then be viewed as a special case of the more general weighted sums 
\begin{equation}\label{E:WeightedSum}
    \sum_{j=1}^Nx_j f(X_j)\,.
\end{equation}
Such weighted nonlinear Gaussian functionals also arise naturally in the analysis of statistical models, for instance in nonlinear regression. We refer the reader to~\cite{ivanv.leonenko.ea:13:limit} for an example, as well as for a central limit theorem for this class of weighted functionals.

Our aim is to establish a Khintchine-type inequality by showing that, under suitable assumptions on the covariance structure and the Hermite expansion of $f$, the $L^p(\Omega)$ norms of~\eqref{E:WeightedSum} are comparable to their $L^2(\Omega)$ norm uniformly in $N$. The following theorem is our main result.

\begin{theorem}\label{thm: Main}
Fix a positive integer $N$ and let $f$ be a function with Hermite rank $d$. For all integers $q\ge d$, denote by the $N\times N$ matrix
\begin{equation*}
    M_q\coloneqq \{\rho(j-k)^q\}_{j,k=1}^N\,.
\end{equation*}
We assume that the matrix $M_q$ is elliptic and bounded as an operator from $\R^N$ to $\R^N$. That is, there exist $0<\lambda_q\le \Lambda_q<\infty$ such that for all $x\in \R^N$,
\begin{equation}\label{A:assumption1}
    \lambda_q|x|^2\le \langle x, M_q x \rangle\le \Lambda_q|x|^2\,,
\end{equation}
where $\langle\cdot,\cdot\rangle$ and $|\cdot|$ are Euclidean space inner product and norm, respectively. 
For any $p>0$, we also assume that
\begin{equation}\label{A:assumption2}
    \sum_{q=d}^{\infty}(p^{\ast}-1)^{\frac{q}{2}}(q!)^{\frac{1}{2}} |a_q| \Lambda_q^{\frac{1}{2}} < \infty\,,
\end{equation}
where $p^{\ast}=\max\{p,4-p\}$.
Then, 
there exist constants $c_p$ and $C_p$ which are independent of $x$ and $N$ such that for $Y_k = f(X_k)$, $k=1,\dots, N$ and any $x = (x_1, \dots, x_N)\in \R^{N}$,  
\begin{align}\label{E:Khintchine}
    c_p \Norm{\sum_{j=1}^{N}x_j Y_j}_{L^p\lr{\Omega}}\le \Norm{\sum_{j=1}^{N}x_j Y_j}_{L^2\lr{\Omega}}\le C_p \Norm{\sum_{j=1}^{N}x_j Y_j}_{L^p\lr{\Omega}}\,.
\end{align}
\end{theorem}
\begin{remark}
    If we assume further that $x\in \R_+^N=[0,\infty)^N$ and $\rho(k)\ge 0$ for all $k$, the lower bound in~\eqref{A:assumption1} is automatically satisfied since
    \begin{align*}
        \langle x, M_q x\rangle\ge \rho(0)|x|^2=|x|^2\,,
    \end{align*}
    for all $x\in \R_+^N$. Also, if $a_q=0$ for some $q$, then we do not need assumption \eqref{A:assumption1} for that $q$. 
\end{remark}
\begin{remark}
    We note that our theorem does not cover the original Khintchine's inequality which is for Rademacher random variables. We believe that the isonormal Gaussian process is not a good setting for Rademacher random variables. A possible approach based on Rademacher functionals, e.g. \cite{NPR}, is left for future research. 
\end{remark}

\begin{remark}
    One of sufficient condition to Condition~\eqref{A:assumption1} is the diagonal-dominance condition of the matrix $M_q$ through Gershgorin's theorem, see e.g.,~\cite[Theorem 6.1.1]{horn.johnson}. If we assume that 
    $$
    \max_{1\leq i\leq N}\sum_{j \neq i} |\rho(i-j)|^d < 1\,.
    $$
    Then by Gershgorin's theorem, the $\Lambda_q$ and $\lambda_q$ for all $q\geq d$ in \eqref{A:assumption1} can be estimated by 
    \begin{equation*}
        0 < 1-\max_{1\leq i \leq N}\sum_{j\neq i}|\rho(i-j)|^q \leq \lambda_q \leq \Lambda_q \leq 1+ \max_{1\leq i \leq N}\sum_{j\neq i}|\rho(i-j)|^q\,
    \end{equation*}
 since $|\rho|\leq 1$. Examples of diagonal dominant covariance function include fast-decaying functions like Gaussian covariance $\rho(k)=e^{-\alpha k^2}$ with sufficiently large $\alpha$.
\end{remark}

\section{Proof of Theorem \ref{thm: Main}}

We first present some background in Wiener chaos expansion and Malliavin calculus, which are needed in our proof. For a detailed account of these materials, see \cite{Nualart}. 

Let $(\Omega,\mathcal{F},\mathbb{P})$ be a complete probability space, and let
\[
    X=\{X(h):h\in\mathfrak{H}\}
\]
be an isonormal Gaussian process over a real separable Hilbert space
$\mathfrak{H}$; that is, $X$ is a centered Gaussian family satisfying
\[
    \mathbb{E}\bigl[X(h)X(g)\bigr]
    =
    \langle h,g\rangle_{\mathfrak{H}},
    \qquad h,g\in\mathfrak{H}.
\]
We assume that $\mathcal{F}$ is generated by $X$, together with the
$\mathbb{P}$-null sets.

We recall the definition of Malliavin derivative. For all $m\ge 1$, denote by $C_b^\infty(\mathbb{R}^m)$ the set of all bounded $C^\infty(\R^m)$ functions, whose partial derivative of all orders are also bounded. Let
$\mathcal{S}$ denote the class of smooth cylindrical random variables of the
form
\[
    F
    =
    f\bigl(X(h_1),\ldots,X(h_m)\bigr),
\]
where $h_1,\ldots,h_m\in\mathfrak{H}$, and
$f\in C_b^\infty(\mathbb{R}^m)$. The Malliavin derivative of $F$ is the
$\mathfrak{H}$-valued random variable
\[
    DF
    =
    \sum_{j=1}^{m}
    \frac{\partial}{\partial x_j} f\bigl(X(h_1),\ldots,X(h_m)\bigr)h_j.
\]
For all integer $k\geq 1$, iterating this construction for $k$ times, we obtain the $k$th-order
Malliavin derivative of $F$ which is given by
\[
    D^kF
    =
    \sum_{j_1,\ldots,j_k=1}^{m}
    \frac{\partial^k}{\partial x_{j_1}\dots\partial x_{j_k}}
    f\bigl(X(h_1),\ldots,X(h_m)\bigr)
    h_{j_1}\otimes\cdots\otimes h_{j_k}.
\]
Due to the summation, this derivative takes value in the symmetric tensor product space $\mathfrak{H}^{\odot k}\subseteq \mathfrak{H}^{\otimes k}$.

For $p\geq 1$ and $k\in\mathbb{N}$, the Sobolev space
$\mathbb{D}^{k,p}$ is defined as the completion of $\mathcal{S}$ with respect
to the norm
\[
    \|F\|_{k,p}^{p}
    =
    \mathbb{E}|F|^{p}
    +
    \sum_{j=1}^{k}
    \mathbb{E}\bigl[\|D^jF\|_{\mathfrak{H}^{\otimes j}}^{p}\bigr].
\]
The operator $D^k$ can be proven to be closable from
$\mathbb{D}^{k,p}$ into
$L^p(\Omega;\mathcal{H}^{\odot k})$ (see e.g.~\cite[Proposition 2.3.4]{nourdin.peccati:12:normal}).

For each $n\geq 0$, let
$\mathcal{H}_n$ denote the $n$th \textit{Wiener chaos} generated by $\{H_n(X(h)): \|h\|_{\mathfrak{H}}=1\}$ where $H_n$ is the $n$th Hermite polynomial~\eqref{E:Hermite_Poly}. Such polynomials are important since they form an orthogonal basis of
$L^2(\mathbb{R},\mu)$, where $\mu$ denotes the standard Gaussian measure. For every $n\geq 1$, there exists a linear
mapping
\[
    I_n:\mathfrak{H}^{\odot n}\longrightarrow \mathcal{H}_n,
\]
called the \textit{multiple Wiener integral} of order $n$. It maps
symmetric tensors of rank one, e.g.~$h^{\otimes n}\in \mathfrak{H}^{\odot n}$, to
\[
    I_n(h^{\otimes n})
    =
    \|h\|_{\mathfrak{H}}^n
    H_n\left(
        \frac{X(h)}{\|h\|_{\mathfrak{H}}}
    \right),
    \qquad h\in\mathfrak{H}\setminus\{0\},
\]
and extends by linearity and continuity to
$\mathfrak{H}^{\odot n}$. In particular, if
$\|h\|_{\mathfrak{H}}=1$, then
\begin{equation}\label{E:Hermite_Integral}
    I_n(h^{\otimes n})=H_n(X(h)),
\end{equation}
which generates $\mathcal{H}_n$.
Different Wiener chaoses are mutually orthogonal in the sense that
\[
    \mathbb{E}\bigl[I_n(f)I_m(g)\bigr]
    =
    \delta_{m,n}\,
    n!\,
    \langle f,g\rangle_{\mathfrak{H}^{\otimes n}},
\]
where $\delta_{m,n}$ refers to the Kronecker delta function. 

 Due to the above orthogonality, we have the \textit{Wiener-It\^{o} decomposition}: For every $F\in {L^2(\Omega)}$, there exists a sequence  $f_n\in\mathfrak{H}^{\odot n}$, such that
\[
    F
    =
    \sum_{n=0}^{\infty}J_nF
    =
    \mathbb{E}[F]
    +
    \sum_{n=1}^{\infty}I_n(f_n),
\]
where $J_n$ is the orthogonal projection onto $\mathcal{H}_n$ and $I_n$ denotes the multiple Wiener integral of order $n$. The Malliavin derivative interacts naturally with the multiple Wiener integral. In particular,
\begin{equation}\label{E:Diff_Integral}
    \mathbb{E}\bigl[D^kI_m(f)\bigr]
    =
    \begin{cases}
        k!f, & m=k,\\
        0,   & m\neq k.
    \end{cases}
\end{equation}
Thus, the expectation of the $k$th Malliavin derivative isolates precisely
the kernel of the $k$th Wiener chaos component.
Therefore, if $F\in \bigcap_k\mathbb{D}^{k,2}$, the corresponding $f_n$ is given explicitly through its $n$th-order Malliavin derivative, which is known as \textit{Stroock's formula},
\begin{equation}\label{E:Stroock}
    f_n
    =
    \frac{1}{n!}\mathbb{E}[D^nF],
    \qquad n\geq 1.
\end{equation}
This is particularly useful when the Malliavin
derivatives of a random variable admit tractable expressions.


Now we are ready to give the proof for Theorem \ref{thm: Main}. 
\begin{proof}

Recall the deterministic function $f$ is given by~\eqref{E:f=Hermite}. We will treat the case $p\geq 2$ and $0<p<2$ separately. 

    \noindent\textbf{Case 1}: $p\ge 2$. In this case, from H\"older's inequality it is obvious that 
    \begin{align}\label{E:L2leLpLargeP}
        \Norm{\sum_{j=1}^{N}x_j Y_j}_{L^2\lr{\Omega}}\le \Norm{\sum_{j=1}^{N}x_j Y_j}_{L^p\lr{\Omega}}\,.
    \end{align}
    It remains to prove the other inequality. From~\eqref{E:Y=f(X)} and~\eqref{E:f=Hermite}, we have
    \begin{align}\label{E:ChaosExpansion}
        \sum_{j=1}^N x_j Y_j
        =\sum_{j=1}^N\sum_{q=d}^\infty a_q H_q\lr{X_j}x_j
        =\sum_{q=d}^\infty I_q\lr{f_{N,q}}\,,
    \end{align}
    where $I_q$ denotes the multiple Wiener integral of order $q$. By Stroock's formula~\eqref{E:Stroock}, and properties~\eqref{E:Hermite_Integral},~\eqref{E:Diff_Integral}, we have
    \begin{align*}
        f_{N,q}
        =\frac{1}{q\,!}\E \left[D^q\sum_{j=1}^N a_q H_q\lr{X_j}x_j\right]
        =\sum_{j=1}^N a_q x_j \varepsilon_j^{\otimes q}\,.
    \end{align*}
    By the hypercontractivity of Wiener chaos, (see, e.g.~\cite[Corollary 2.8.14]{nourdin.peccati:12:normal})
    \begin{align*}
        \Norm{I_q\lr{f_{N,q}}}_{L^p\lr{\Omega}}\le \lr{p-1}^{\frac{q}{2}}\Norm{I_q\lr{f_{N,q}}}_{L^2\lr{\Omega}}\,.
    \end{align*}
    Hence, from~\eqref{E:ChaosExpansion},
    \begin{align}\label{E:xYHyper}
        \NormO{\sum_{j=1}^Nx_jY_j}{p}
        \le \sum_{q=d}^\infty \NormO{I_q\lr{f_{N,q}}}{p}
        \le \sum_{q=d}^\infty \lr{p-1}^{\frac{q}{2}}\NormO{I_q\lr{f_{N,q}}}{2}\,.
    \end{align}
    For each $q\ge d$, 
    \begin{align*}
        \NormO{I_q\lr{f_{N,q}}}{2}^2
        &=q\,!\Norm{\sum_{j=1}^Na_qx_j \varepsilon_j^{\otimes q}}_{\mathfrak{H}^{\otimes q}}^2 =q\,!a_q^2\sum_{j,k=1}^Nx_jx_k\rho\lr{j-k}^q\,.
    \end{align*}
    Recall by assumption~\eqref{A:assumption1},
    \begin{align*}
        \sum_{j,k=1}^Nx_jx_k\rho\lr{j-k}^q
         \le  \Lambda_q\sum_{j=1}^N|x_j|^2\,.
    \end{align*}
    Therefore,
    \begin{align}\label{E:Lptol2}
        \NormO{\sum_{j=1}^Nx_jY_j}{p}
        \le \sum_{q=d}^\infty \lr{p-1}^{\frac{q}{2}}\lr{q\,!}^{\frac{1}{2}}\left|a_q\right|\Lambda_q^{\frac{1}{2}}\lr{\sum_{j=1}^N|x_j|^2}^{\frac{1}{2}}\,.
    \end{align}
    Again from the uniform ellipticity of $M_q$~\eqref{A:assumption1}, we have
    \begin{align}\label{E:OffDiag}
        \sum_{j,k=1}^Nx_jx_k\rho\lr{j-k}^q
         \geq  \lambda_q\sum_{j=1}^N|x_j|^2\,.
    \end{align}
 Using the orthogonality of the Wiener chaoses,  we obtain the lower bound
    \begin{align}\label{E:L2Lower}
        \NormO{\sum_{j=1}^Nx_jY_j}{2}^2
    =\sum_{q=d}^\infty \NormO{I_q\lr{f_{N,q}}}{2}^2 \geq  \sum_{q=d}^\infty q\,!a_q^2\lambda_q\sum_{j=1}^N |x_j|^2\,.
    \end{align}
    Combining~\eqref{E:Lptol2} and~\eqref{E:L2Lower} gives
    \begin{align*}
        \NormO{\sum_{j=1}^Nx_jY_j}{p}
        \le
        \frac{\sum_{q=d}^\infty\lr{p-1}^{\frac{q}{2}}\lr{q!}^{\frac{1}{2}}\left|a_q\right|\Lambda_q^{\frac{1}{2}}}
        {\lr{\sum_{q=d}^\infty q\,!a_q^2 \lambda_q}^{\frac{1}{2}}}
        \NormO{\sum_{j=1}^Nx_jY_j}{2}\,.
    \end{align*}
    Setting
    \begin{align}\label{E:Cp}
        \tilde{C}_p=
        \frac{\sum_{q=d}^\infty\lr{p-1}^{\frac{q}{2}}\lr{q!}^{\frac{1}{2}}\left|a_q\right|\Lambda_q^{\frac{1}{2}}}
        {\lr{\sum_{q=d}^\infty q\,!a_q^2\lambda_q}^{\frac{1}{2}}}\,,
    \end{align}
    we have
    \begin{align}\label{E:LpleL2LargeP}
        \Norm{\sum_{j=1}^{N}x_j Y_j}_{L^p\lr{\Omega}}\le \tilde{C}_p\Norm{\sum_{j=1}^{N}x_j Y_j}_{L^2\lr{\Omega}}\,.
    \end{align}
    Combining~\eqref{E:L2leLpLargeP} with~\eqref{E:LpleL2LargeP} verifies~\eqref{E:Khintchine} for $p\ge2$, with $c_p = \tilde{C}_p^{-1}$ and $C_p=1$. 

    \noindent\textbf{Case 2}: $0<p<2$. In this case, H\"older's inequality gives
    \begin{align}\label{E:LpleL2SmallP}
        \Norm{\sum_{j=1}^{N}x_j Y_j}_{L^p\lr{\Omega}}\le \Norm{\sum_{j=1}^{N}x_j Y_j}_{L^2\lr{\Omega}}\,.
    \end{align}
    For the opposite direction, let $Z$ be a random variable with finite $L^{4-p}\lr{\Omega}$ norm. Then
    \begin{align*}
        \NormO{Z}{2}^2
        &=\E |Z|^2
        =\E \lr{|Z|^{\frac{p}{2}}|Z|^{\frac{4-p}{2}}}\\
        &\le \lr{\E|Z|^p}^{\frac{1}{2}}\lr{\E|Z|^{4-p}}^{\frac{1}{2}}\\
        &= \NormO{Z}{p}^{\frac{p}{2}}\NormO{Z}{4-p}^{\frac{4-p}{2}}\,.
    \end{align*}
    We apply this estimate to $ Z=\sum_{j=1}^N x_jY_j$, 
 and use the result proved in Case~1 for $\|Z\|_{L^{4-p}(\Omega)}$, 
 it implies that
    \[
        \NormO{Z}{4-p}\le C_{4-p}\NormO{Z}{2}<\infty .
    \]
    Hence
    \begin{align*}
        \NormO{Z}{2}^2
        \le C_{4-p}^{\frac{4-p}{2}}\NormO{Z}{p}^{\frac{p}{2}}\NormO{Z}{2}^{\frac{4-p}{2}}\,.
    \end{align*}
    If $\NormO{Z}{2}=0$, there is nothing to prove. Otherwise, multiplying both sides by $\NormO{Z}{2}^{\frac{p-4}{2}}$ yields
    \begin{align*}
        \NormO{Z}{2}^{\frac{p}{2}}\le C_{4-p}^{\frac{4-p}{2}}\NormO{Z}{p}^{\frac{p}{2}}\,,
    \end{align*}
    and therefore
    \begin{align*}
        \NormO{Z}{2}\le C_{4-p}^{\frac{4-p}{p}} \NormO{Z}{p}\,.
    \end{align*}
    Returning to $Z=\sum_{j=1}^N x_jY_j$ and setting $C_p=\tilde{C}_{4-p}^{(4-p)/p}$ where the constant $\tilde{C}_p$ is given by~\eqref{E:Cp}, we obtain
    \begin{align*}
        \Norm{\sum_{j=1}^{N}x_j Y_j}_{L^2\lr{\Omega}}\le c_p\Norm{\sum_{j=1}^{N}x_j Y_j}_{L^p\lr{\Omega}}\,.
    \end{align*}
    Combining this inequality with~\eqref{E:LpleL2SmallP} completes the proof for $0<p<2$ with $C_p = \tilde{C}_{4-p}^{\frac{4-p}{p}}$ and $c_p=1$. 

    The proof of Theorem \ref{thm: Main} is completed. 
\end{proof}


\section*{Acknowledgements}
The authors would like to thank Raluca Balan for helpful discussions.


\begin{thebibliography}{1}




\bibitem{BM}
P. Breuer and P. Major,
Central limit theorems for nonlinear functionals of Gaussian fields,
\emph{J. Multivariate Anal.} \textbf{13} (1983), no.~3, 425--441.



\bibitem{Havrilla}
A. Havrilla,
A survey of Khintchine type inequalities for random variables,
Master's thesis, Department of Mathematical Sciences, Carnegie Mellon University, 2021.

\bibitem{horn.johnson}
R.~A. Horn and C.~R. Johnson,
\emph{Matrix Analysis},
Cambridge University Press, Cambridge, 1985; MR0832183.

\bibitem{ivanv.leonenko.ea:13:limit}
A.~V. Ivanov, N. Leonenko, M.~D. Ruiz-Medina and I.~N. Savich,
Limit theorems for weighted nonlinear transformations of Gaussian stationary processes with singular spectra,
\emph{Ann. Probab.} \textbf{41} (2013), no.~2, 1088--1114.

\bibitem{Khintchine}
A. Khintchine,
Über dyadische Brüche,
\emph{Math. Z.} \textbf{18} (1923), no.~1, 109--116.

\bibitem{MS}
C. Muscalu and W. Schlag,
\emph{Classical and Multilinear Harmonic Analysis. Vol. I},
Cambridge Stud. Adv. Math., 137,
Cambridge University Press, Cambridge, 2013. xviii+370 pp.


\bibitem{nourdin.peccati:12:normal}
I.~Nourdin and G.~Peccati,
\emph{Normal Approximations with Malliavin Calculus: From Stein's Method to Universality},
Cambridge Tracts in Mathematics, vol.~192,
Cambridge University Press, Cambridge, 2012.
DOI: \href{https://doi.org/10.1017/CBO9781139084659}{10.1017/CBO9781139084659}.

\bibitem{NPR}
I. Nourdin, G. Peccati and G. Reinert,
Stein's method and stochastic analysis of Rademacher functionals,
\emph{Electron. J. Probab.} \textbf{15} (2010), no.~55, 1703--1742.

\bibitem{Nualart}
D. Nualart,
\emph{The Malliavin Calculus and Related Topics},
2nd ed.,
Probability and its Applications,
Springer, Berlin, 2006.

\end{thebibliography}
\end{document}